\documentclass[12pt]{amsart}
\usepackage{amsmath, latexsym, amsfonts, amssymb, amsthm, amscd}

\DeclareMathOperator{\wlim}{w-lim}

\newcommand{\kap}{\mathrm{cap}}

    \def\beq{\begin{eqnarr}}
    \def\eeq{\end{eqnarr}}
    \def\beqq{\begin{eqnarr*}} 
    \def\eeqq{\end{eqnarr*}} 

\newtheorem {definition}{Definition}

\newtheorem {example}[definition]{Example}

\newtheorem {thm}[definition]{Theorem}

\newtheorem {corollary}[definition]{Corollary}

\newcommand{\Z}{\mathbb{Z}}

\newcommand{\cQ}{{\mathcal Q}}
\newcommand{\cP}{{\mathcal P}}

\newcommand{\N}{\mathbb{N}}

\newcommand{\R}{\mathbb{R}}

\title[Random Interlacements via Kuznetsov Measures]{Random Interlacements via Kuznetsov Measures}

\author{Steffen Dereich}
\address{Steffen Dereich, Institut f\"ur Mathematische Statistik, Westf\"alische Wilhelms-Universit\"at M\"unster, Orl\'eans-Ring 10, 48149 M\"unster, Germany}
\email{}

\author{Leif D\"oring}
\address{Leif D\"oring, Departement Mathematik, ETH Z\"urich, R\"amistrasse 10, 8092 Z\"rich, Switzerland }
\email{}
\thanks{
}

\begin{document}

\maketitle

\begin{abstract}
The aim of this note is to give an alternative construction of interlacements - as introduced by Sznitman - which makes use of classical potential theory. In particular, we outline that the intensity measure of an interlacement is known   in probabilistic potential theory under the name "approximate Markov chain" or "quasi-process".
We provide a simple construction of random interlacements through (unconditioned) two-sided Brownian motions (resp.  two-sided random walks) involving Mitro's general construction of Kuznetsov measures and a Palm measures relation due to Fitzsimmons. In particular, we show that random interlacement is a Poisson cloud (`soup') of two-sided random walks (or Brownian motions) started in Lebesgue measure and restricted on being closest to the origin at time $T\in [0,1]$ - modulus time-shift.
\end{abstract}

\section{Introduction}

Since the seminal article of Sznitman \cite{S1} the subject of random interlacements has been very active.  Roughly speaking, a random interlacement is a random `soup' of trajectories indexed by the real time-axis such that a finite Poisson distributed number of trajectories can be observed when looking into an arbitrary compact set. For transient processes, such as simple random walk (SRW) on $\Z^d$, $d\geq 3$, or Brownian motion (BM) on $\R^d$, $d\geq 3$, the construction is non-trivial and done via a projective limit procedure. The original motivation to study random interlacements stems from connections to earlier studied objects such as the random walks in tori or cylinders and the Gaussian free field (see for instance \cite{W08}, \cite{S2}) and a correlation structure that allows a detailed study of percolative properties of the vacant set not covered by the trajectories (see for instance \cite{S3}, \cite{SS09}). \smallskip

The aim of the present note is to look at the model of random interlacements itself from a different angle rather than proving further properties. We explain why interlacements are very natural objects that existed in the literature long before under the name quasi-process and resulted in general excursion theory. Those were introduced by Hunt \cite{Hunt} in the probabilistic study of Martin entrance boundaries at that time under the name approximate Markov chain. Quasi-processes are variants of a Markov process with random birth and death in such a way that the occupation time measure is excessive for the Markov process. \smallskip

There is a reason why the interpretation of random interlacements as quasi-processes is interesting: Quasi-processes are linked one-to-one to so-called Kuznetsov measures and under duality assumptions a simple two-sided constructions of Kuznetsov measures exists due to Mitro \cite{Mitro}. We are not aware of such a simple construction for quasi-processes but using that quasi-process can be obtained as Palm measures of Kuznetsov measures Mitro's construction can be employed to construct quasi-processes in a direct way. We apply this idea to random interlacements and derive a representation via two-sided SRW or BM. \smallskip

\textbf{From now on all arguments are given for Brownian interlacements, all arguments work analogously for random walk interlacements.}\smallskip

We should note that all we do works also in more generality without any changes: For instance the underlying motion can be replaced by a non-simple RW or an arbitrary L\'evy process.\smallskip

After writing the article we  learnt from Jay Rosen that in the recent article \cite{R} he also used a relation between random interlacements and quasi-processes. Rosen introduced  L\'evy-driven random interlacements as a `soup' of L\'evy quasi-processes (i.e. Poisson point measure with quasi-process intensity) and used tools from probabilistic potential theory (in particular Revuz measures) for the analysis of interesection local times.  \smallskip

\textbf{Outline:} In Section 2 we recall the definition of random interlacements due to Sznitman. An overview of some results and objects from probabilistic potential theory is presented in Section 3. Finally, in Section 4 we explain how random interlacements can be constructed from unconditioned two-sided BM.

\section{Random Interlacements}\label{Sec2}
In this section we give the definition of random interlacements following \cite{S2}. As mentioned above, we stick to Brownian interlacements; the notation and results presented here are analogous in the discrete setting.\smallskip

 Denote by $W_+$ the subspace of $C(\R_+, \R^d)$, $d\in\{ 3,4,\dots\}$, of continuous $\R^d$-valued trajectories tending to infinity at infinity and by $W$ the subspace of $C(\R, \R^d)$ consisting of continuous trajectories indexed by the real-line which diverge at both ends. The canonical process will be denoted $(\omega_t)_{t\geq 0}$ (resp. $(\omega_t)_{t\in \R}$) and the shift-operators by $\sigma_t, t\geq 0$, (resp. $\sigma_t, t\in\R$).
 We work with the canonical $\sigma$-algebras $\mathcal W_+$ and $\mathcal W$ generated by the shift-operators. The first-hitting time of a closed set $F$ shall be denoted by $H_F$. The same symbol  
$H_F$ will be used for $W_+$ and $W$ without causing confusion.\smallskip
	
	Let $W^*$ denote the set $W$ modulus identification by time-change, that is $W^*=W \slash \sim$ with $\omega \sim \omega'$ if there is some $t\in \R$ such that $\omega(s)=\omega'(s+t)$ for $s\in\R$. Let $\pi$ denote the projection from $W$ to $W^*$ and $\mathcal W^*$ the induced $\sigma$-algebra on $W^*$. For any compact set $B\subset \R^d$ define $W_B$ (resp. $W^\ast_B$) the set of trajectories hitting $B$ (resp. equivalence classes with one (and then all) representatives hitting $B$).\smallskip
	
 For $y\in \R^d$ we denote by $P_y$ the law of Brownian motion started in $y$ and for a compact set $B$ not containing $y$, $P_y^B(\cdot)=P_y(\cdot\,|\, H_B=\infty)$. We extend the definition of $P_y^B$ for  $y\in \partial B$ by taking weak limits
$$
P_y^B=\wlim_{x\to y, x\in B^c} P_x^B
$$ 
provided that the limit exists. As is well known this is always the case if $B$ is a ball.
The law $P_y^B$ will be called the Brownian motion conditioned on avoiding $B$.\smallskip

For a compact set $B$ let $e_B$ denote the Brownian equilibrium measure on $B$ with capacity $\kap (B)=e_B(B)$ (see for instance \cite{MP}). There are several descriptions of $e_B$ but all we need is a relation to the harmonic measure (from infinity):
\begin{align}\label{hhh}
	\wlim_{|y|\to \infty} P_y(W_{H_B}\in dy\,|\, H_B<\infty)= \frac{e_B(dy)}{\kap (B)}
\end{align}
and the invariance property
\begin{align}\label{hh}
	P_{e_{B'}}(H_B \in dx)=e_B(dx),\quad \text{ for } B\subset B',
\end{align}
which implies $P_{e_{B'}}(H_B <\infty)=\kap (B)$.\smallskip

The crucial ingredients of random interlacements are the measures $Q_B$ on $\mathcal W$ defined as follows:
\begin{align*}
	Q_B\big((\omega_{-t})_{t\geq  0}\in A, \omega_0\in dy, (\omega_t)_{t\geq 0} \in A'\big)= P_y^B(A)\, e_B(dy)\, P_y(A'),\quad \forall A, A'\in \mathcal W_+,
\end{align*}
for a closed ball $B\subset \R^d$.
That is, from a starting point $y\in \partial B$ sampled according to $e_B$ a Brownian motion is started in positive time-direction and independently a conditioned Brownian motion is started in negative time-direction.
For a general compact set $K\subset B$ define $Q_K$ by restriction:
\begin{align}\label{res}
	Q_K:= \sigma_{H_K}\circ (1_{H_K<\infty} Q_B).
\end{align}
 By definition the $Q_K$ are concentrated on paths hitting the compact set $K$ for the first time at time $0$. One then defines the measures $Q^*_K$ on $(W_K^*, \mathcal W_K^*)$ by projecting $Q_K$  
onto $W_K^*$:
$$Q^*_K:=\pi \circ Q_K.$$
\begin{definition}\label{RI}
	A measure $\nu$ on $(W^*, \mathcal W^*)$ such that for compact sets $K\subset \R^d$ the restrictions to $(W^*_K,\mathcal W^*_K)$ are $Q^*_K$ (i.e. $1_{W^*_K} \nu = Q^*_K$) is called interlacement intensity.
\end{definition}
Brownian interlacement intensities have been constructed in \cite{S2} by an abstract projective limit procedure. The proof is similar to the original argument from \cite{S1} for random walks.
\begin{thm}[Sznitman \cite{S2}]\label{Th}
	The interlacement intensity exists and is unique.
\end{thm}
 We will give a more direct and perhaps easier construction via unconditioned two-sided Brownian motion.\smallskip
 
  It was shown in Remark 2.3 of \cite{S2} that $\nu$ has a remarkable time-inversion duality, i.e. the measure $\hat \nu$ obtained by reversing the time-direction is identical to $\nu$. As we shall see in Corollary \ref{co} below this is also classical in probabilistic potential theory.\smallskip
  
Finally, we should also mention the definition of the object of main interest:
 \begin{definition}
	A Poisson point process $\xi$ on $W^*$ with intensity measure $\alpha\nu$ is called Brownian random interlacement at level $\alpha>0$.
\end{definition}
In Section 4 we shall use a standard relation for Poisson point processes to give a simple construction of random interlacements in terms of conditioned two-sided Brownian motion.

\section{Elements of Probabilistic Potential Theory}
Suppose $X=(X_t)_{t\geq 0}$ is a Feller process on a locally compact separable metric  space~$E$ with transition semigroup $(P_t)_{t\geq 0}$. We will assume that $X$ is  transient meaning that there exists $f>0$ measurable with $\int_0^\infty P_t f\,dt<1$. As usual the state space $E$ is compactified by a cemetery state $\partial$.	
\smallskip
	
A central aim of probabilistic potential theory is the probabilistic analysis of excessive measures, i.e. $\sigma$-finite measures $\eta$ with $\eta P_t(dy) := \int \eta(dx) P_t(x,dy) \leq \eta(dy)$ for $t\geq 0$. 	For the study of excessive measures two versions of extended Markov processes have been introduced: {\it Kuznetsov measures} and {\it quasi-processes}. Both concepts also turned out valuable in the general theory of stochastic processes, such as in excursion theory and in the theory of reversing Markov processes.\smallskip
	
	 Let us first introduce Kuznetsov measures and consider the space $\Omega$ of RCLL (right continuous with left limits) trajectories $\omega:\R\to E\cup \{\partial\}$ that are $E$-valued on some open interval $(\alpha(\omega),\beta(\omega))$, taking the value $\partial$ outside $[\alpha(\omega),\beta(\omega))$. 
Equip $\Omega$ with the canonical $\sigma$-fields $\mathcal G_t=\sigma(\omega_s, s\leq t)$ and $\mathcal G=\sigma(\omega_t, t\in\R)$. Then a measure on $(\Omega, \mathcal G)$ is called Kuznetsov measure for $(P, \eta)$ if $\cQ$ does not charge the trivial path $\omega\equiv \partial$ and $\cQ$ has finite dimensional marginals
\begin{align*}
	&\quad \cQ\big(\alpha(w)<t_1, w_{t_1}\in d x_1, \cdots ,w_{t_n}\in dx_n, t_n<\beta(w)\big)\\
	&=\eta(dx_1)P_{t_2-t_1}(x_1,dx_2)\cdots P_{t_n-t_{n-1}}(x_{n-1},dx_n). 
\end{align*}
Under a Kuznetsov measure $\cQ$ the canonical process is also called a Markov process with random birth and death: $\alpha(\omega)=\inf\{t:\omega_t\in E\}$ is called the time of birth, $\beta(\omega)=\sup \{t:\omega_t\in E\}$ the time of death and $\zeta=\beta-\alpha$ the life-time. The life-time is well-defined and the strong Markov property holds under $\cQ$.\smallskip

\begin{thm}[Kuznetsov, \cite{Kuznetsov}]
	If $\eta$ is excessive for $(P_t)$, then there is a unique Kuznetsov measure for $(P, \eta)$.
\end{thm}
There are different proofs of the theorem, the original proof due to Kuznetsov \cite{Kuznetsov} is based on Kolmogorov's extension theorem; a simple construction under duality assumptions is recalled below.\smallskip

To get an idea of the need for finite and infinite birth times let us mention a probabilistic representation (due to Fitzsimmons, Maisonneuve \cite{FitzsimmonsMaisonneuve}) of the Riesz decomposition of an excessive measure $\eta$ into the sum $\eta=\eta_I+\eta_P$, where $\eta_I$ is an invariant measure (i.e. $\eta P_t=\eta$ for all $t\geq 0$) and $\eta_P$ is a purely excessive measure (i.e. $\eta_P P_t (f) \downarrow 0$ as $t\to\infty$ for any $f\geq 0$ with $\eta_P(f)<\infty$). The probabilistic representation of $\eta_I, \eta_P$ is obtained by separating paths of $\cQ$ according to infinite birth time and finite birth time: With
\begin{align*}
\cQ_I(\cdot)= \cQ( \cdot\,,\, \alpha=-\infty) \ \text{ and } \ \cQ_P(\cdot\,,\, -\infty <\alpha)
\end{align*}
one has for $t\in\R$
\begin{align*}
	\eta_I(\cdot)= \cQ_I(\omega_t \in \cdot\, ,\, t <\beta) \ \text{ and } \ \eta_P (\cdot)=\cQ_P(\omega_t\in \cdot\,,\, \alpha<t<\beta).
\end{align*}

We now turn to the second type of generalized Markov process that goes back to Hunt \cite{Hunt} in discrete time under the original name approximate Markov chain and has been extended by Weil \cite{Weil1} under the name quasi-process to continuous time.
Consider again the space $(\Omega, \mathcal G)$ of trajectories equipped with the canonical $\sigma$-field. 
An $\R\cup\{\infty\}$-valued random variable $S$ is called stationary time, if $S=S\circ \sigma_t+t$ and $\alpha\leq S<\beta$ on $\{S<\infty\}$. Further, it is called intrinsic, if it is almost everywhere finite. For quasi-processes one does not work on the entire $\sigma$-field $\mathcal G$ but instead on the $\sigma$-field $\mathcal A$ of $\sigma_t$-invariant events [Note: this corresponds to taking the quotient space in Section \ref{Sec2}].
If $\eta$ is excessive for $(P_t)$, then a measure $\cP$ on $(\Omega,\mathcal A)$ is called a quasi-process for $(P,\eta)$ if
\begin{align*}
	\eta(\cdot)=\cP\left(\int_{\alpha(\omega)}^{\beta(\omega)} 1_{\cdot}(\omega_s)\,ds\right)
\end{align*}
and for any intrinsic stopping time $S$ under $\cP_{|_{S\in \R}}$ the forward in time process $(\omega_{S+t})_{t\geq 0}$ is strong Markov with transitions $(P_t)$.\smallskip

Both, Kuznetsov measure and quasi-process have advantages and disadvantages in the study of excessive measures but surprisingly both concepts are equivalent via the notion of Palm measures as noticed by Fitzsimmons:\begin{thm}[Fitzsimmons \cite{Fitzsimmons}]\label{Tf}
	If $\cQ$ is a Kuznetsov measure for $(P,\eta)$ and $S$ is a finite stationary time then the Palm measure of $\cP$ and $S$ 
	\begin{align}\label{palm}
		\cP(A):=\cQ(A, S\in [0,1]),\quad A\in \mathcal A,
	\end{align}
	is a $(P,\eta)$ quasi-process and $\cP$ is independent of the choice of $S$. Conversely, if $\cP$ is a quasi-process for $(P,\eta)$, then
	\begin{align*}
		\cQ(f):= \cP\left(\int_{-\infty}^\infty f\circ \sigma_t \,dt\right),\quad f\text{ is }\mathcal G\text{-measurable},
	\end{align*}	
	defines a Kuznetsov measure for $(P,\eta)$.
\end{thm}
In words: Given a Kuznetsov measure and a stationary time $S$, restricting to trajectories with $S\in [0,1]$ and forgetting about the time parametrization results in a quasi-process. Conversely, given a quasi-process defined on the invariant $\sigma$-field, then one obtains the corresponding Kuznetsov measure by independently attributing to each path a reference time.\smallskip

A simple consequence is the existence and uniqueness of quasi-processes, deduced from the analogue statement for Kuznetsov measures:
\begin{corollary}
	Given an excessive measure $\eta$ for a Feller transition semigroup $(P_t)$ there is a unique quasi-process for $(P,\eta)$. 
\end{corollary}
It depends on the problem whether it is more useful to work with the $(P,\eta)$ Kuznetsov measure or quasi-process. A general advantage of Kuznetsov measures is employed in the next section: Under a duality assumption there is always a simple construction for Kuznetsov measures as (unconditioned) two-sided process, whereas there is no simple construction for quasi-processes.

\section{Random Interlacements through Mitro's Construction}

A particularly simple construction of Kuznetsov measures for Markov processes in duality with respect to an invariant measure $\eta$ is due to to Mitro. Suppose $X$ is a Feller process on $E$ with transition semigroup $(P_t)$ and $X$ is in duality to a second Feller process $\hat X$ with respect to an invariant measure $\eta$, that is, for all $h,g$ bounded and measurable,
\begin{align*}
	E^\eta[h(X_t)g(X_0)]=\hat E^\eta\big[h(\hat X_0)g(\hat X_t)\big],\quad \forall t\geq 0,
\end{align*}
or, equivalently, in the more usual notation $\big\langle P_th,g\rangle_\eta=\big\langle h,\hat P_tg\rangle_\eta$. Hunt's
 achievement was to relate duality and time-reversal in the setting of quasi-processes:
\begin{thm}[Hunt \cite{Hunt}]\label{Hunt} 
	If $\cP$ is the quasi-process for $(P,\eta)$, then the measure $\hat \cP$ obtained under time-reversal and switching to the RCLL version 
 is the quasi-process for $(\hat P, \eta)$.
\end{thm}
Time-reversal for dual processes also holds for Kuznetsov measures and is reflected in an important construction: 
\begin{thm}[Mitro \cite{Mitro}]
The Kuznetsov measure $\cQ$  for $(P,\eta)$ is the unique measure with finite dimensional marginals
	\begin{align*}
		\int_E  \hat P_x[\hat X_{s_1}\in dy_1,\cdots ,\hat X_{s_l}\in dy_l]\,\eta(dx)\, P_x[ X_{t_1}\in dx_1,\cdots  ,X_{t_k}\in dx_k]
	\end{align*}
	at the times $s_l<\cdots <s_1<0\leq t_1<\cdots t_k$.
\end{thm}
In words, to build $\cQ$ one samples the invariant measure $\eta$ at time $0$, and from the outcome starts an independent copy of $X$ to the right and an independent copy of the dual $\hat X$ to the left. \smallskip

Mitro's theorem allows to give a simple argument for Hunt's time-reversal theorem:
\begin{proof}[Proof of Theorem \ref{Hunt}]
 If $X$ and $\hat X$ are in duality with respect to $\eta$, then Mitro's construction explains how to construct $\cQ$ and how to construct $\hat \cQ$. In particular, this  shows that one being the image under time-reversal of the other. Also note that stationary times for $\cQ$ are stationary for $\hat \cQ$ as well, thus, Fitzsimmons' representation (Theorem ~\ref{Tf}) implies the claim. 
\end{proof}

\begin{example}
	A $d$-dimensional Brownian motion is in duality with itself with respect to  Lebesgue measure $\lambda^d$ on $\R^d$.
\end{example}
Since Brownian motion is in duality to itself, Mitro's theorem gives a simple construction of the Kuznetsov measure of a Brownian motion with Lebesgue invariant measure. Taking the Palm measure with respect to any stationary time yields a simple representation of the corresponding quasi-process which, in fact, is Sznitman's interlacement intensity $\nu$:
\begin{thm}\label{t}
	(a) Suppose $X$ is a $d$-dimensional BM and $\lambda^d$ the Lebesgue measure on $\R^d$, then the Brownian interlacement intensity is the quasi-process $\cP$ for $(P,\beta \lambda^d)$ for some $\beta>0$.\\ 
	(b) The Brownian interlacement intensity $\nu$ is obtained by taking a two-sided BM started in Lebesgue measure, restricting an arbitrary intrinsic (i.e.\ finite and stationary)  time to $[0,1]$ and identifying paths that are translations of each other.
	\end{thm}
	Recall from the discussion of the Riesz decomposition above that $\alpha=-\infty$ for  $\cQ$-almost all trajectories since $\lambda^d$ is invariant. The transience of Brownian motion on~$\R^d$ for  $d\geq 3$ then implies that first entrance and last exit times for compact sets are not intrinsic. Instead one can choose the intrinsic time $S=\text{argmin}_{t\in\R} \{|\omega_t|\}$. We emphasize once more that the choice of the intrinsic time is irrelevant. 

\begin{proof}[Proof of Theorem \ref{t}]
	Only (a) needs a proof, (b) is then a consequence of Mitro's construction and Fitzsimmons' theorem.\smallskip
	
	It is enough to show that for some $\beta>0$ the (unique) quasi-process for $(P,\beta \lambda^d)$ has the defining properties of the (unique) interlacement intensity. The multiple $\beta$ will be chosen in due course. \smallskip
	
(i) 
In this first part of the proof we calculate the first-hitting distribution $\cP( \omega_{H_B}\in \cdot)$ of a compact set $B$. Let us first show that $\cP(H_B<\infty)<\infty$. By monotonicity we can restrict  attention to $B$ with strictly positive capacity.
By definition $\cP$ is the quasi-process for $(P,\beta \lambda^d)$, hence, the strong Markov property yields
	\begin{align*}
		\infty>\beta \lambda^d(B)&=\cP\left(\int_{-\infty}^\infty 1_B(\omega_t)\,dt\right)\\
		&=\cP\left(\int_{H_B}^\infty 1_B(\omega_t)\,dt,H_B<\infty\right)\\
		&=\int_B \cP(\omega_{H_B}\in dx, H_B<\infty) E^x\left[\int_0^\infty 1_{B}(X_t)\,dt\right]\\
		&\geq \cP(H_B<\infty) \inf_{x\in B}E^x\left[\int_0^\infty 1_{B}(X_t)\,dt\right].
	\end{align*}
	The infimum is strictly positive, thus, $\cP(H_B<\infty)<\infty$.\smallskip

	From the strong Markov property we obtain, for any $N\in \N$ large enough with $B\subset B_{N}=\{x:|x|\leq N\}$,
	\begin{align*}
		&\quad \cP\left(\omega_{H_B}\in dx, H_B<\infty\right)\\
		&= \cP\left(P_{\omega_{H_{B_N}}}\left[X_{H_B}\in dx\,|\, H_B<\infty\right]P_{\omega_{H_{B_N}}}\left[H_B<\infty\right]1_{H_{B_N}<\infty}\right).
	\end{align*}
	Using that by construction $\cP$ a.a. sample paths have left-limits $\infty$ in norm and additionally \eqref{hhh} we obtain
	\begin{align*}
		 \cP(&\omega_{H_B}\in dx, H_B<\infty)\\
		&= \lim_{N\to\infty} \cP\left(\omega_{H_B}\in dx, H_B<\infty, H_{B_N}<\infty\right)\\
		&=\lim_{N\to\infty}\cP\left(P_{\omega_{H_{B_N}}}\left[X_{H_B}\in dx\,|\, H_B<\infty\right]P_{\omega_{H_{B_N}}}(H_B<\infty)1_{H_{B_N}<\infty}\right)\\
		&=\frac{e_B(dx)}{\kap (B)}\lim_{N\to\infty}\cP\left(P_{\omega_{H_{B_N}}}\left[H_B<\infty\right]1_{H_{B_N}<\infty}\right)\\
		&=\frac{e_B(dx)}{\kap (B)}\cP\left(H_B<\infty\right),
	\end{align*}
	where $e_B(dx)$ is the equilibrium measure on $B$. We now claim that~$\cP$ can be normalized so that $\cP(H_{B}<\infty)=e_B(B)$ for all compact $B$ with strictly positive capacity.
 First, by the strong Markov property and \eqref{hh}, we obtain for $B\subset B'$,
	\begin{align*}
		\cP(H_B<\infty)&=\cP\big( P_{\omega_{H_{B'}}}( H_B<\infty)   1_{H_{B'}} <\infty\big)\\
		&=\int \cP(\omega_{H_{B'}}\in dx , H_{B'} <\infty)  P_x( H_B<\infty)  \\
		&=\int \frac{e_{B'}(dx)}{\kap (B')}\cP\left(H_{B'}<\infty\right)  P_x( H_B<\infty)\\
		&=\cP\left(H_{B'}<\infty\right)\frac{1}{\kap (B')} \, P_{e_{B'}}( H_B<\infty)\\
		&=\cP\left(H_{B'}<\infty\right)\frac{1}{\kap (B')} \, \kap (B)
	\end{align*}
	which proves that for general compact sets $B\subset B'$ with strictly positive capacity
	\begin{align}\label{q}
		\cP(H_{B}<\infty)=\kap (B)\quad \Longleftrightarrow \quad \cP(H_{B'}<\infty)=\kap (B').
	\end{align}
	 Normalizing $\cP$ (i.e. choosing $\beta$ appropriately) such that $\cP(H_{B_1}<\infty)=\kap (B_1)$, equivalence \eqref{q} implies $\cP(H_{B}<\infty)=e_B(B)$ for all $B$ compact. Hence, we have proved
	 \begin{align*}
		 \cP\left(\omega_{H_B}\in dx, H_B<\infty\right)=e_B(dx),\quad \forall B\text{ compact},
	 \end{align*}
	 for a suitable choice of $\beta>0$. Hence, the hitting distribution is the one needed for Definition \ref{RI}.\smallskip
	
	(ii) Since $H_B$ is a stationary stopping time, the forward in time process after $H_B$ under $\cP(\cdot, H_B<\infty)$ is a BM as needed for Definition \ref{RI}.\smallskip
	
(iii) Finally, we need to determine the backwards dynamics. Let  $B$ denote a closed ball and denote by $\hat H_B$ the last exit time from $B$.
Using that Brownian motion started in Lebesgue measure is self-dual we conclude that under~$\cP$  the time-reversed  process $(\omega_{H_B-t})_{t\geq 0}$ up to time $H_B$ has the same `distribution' as the forwards in time process  $(\omega_{\hat H_B+t})_{t\geq 0}$ after time $\hat H_B$. Recall that 
$(\omega_{H_B+t})_{t\geq 0}$ is a Brownian motion started in the finite measure  $\cP(\omega_{H_B}\in \cdot, H_B<\infty)$. Hence, by the last exit decomposition of Getoor \cite{Getoor}, Theorem 2.12, we know that under $\cP$  the pre-exit process and the post-exit processes of $(\omega_{H_B+t})_{t\geq 0}$ are independent conditionally on the point of last exit with the post-exit process being  Brownian motion conditioned on not returning to $B$.\smallskip

(i)-(iii) show that for any closed ball $B$ the defining property $$\cP(\cdot, H_B<\infty)=Q^*_B(\cdot)$$ holds.
Using the definition \eqref{res} of $Q_K$ for a compact set $K\subset B$, the interlacement property for $\cP$ is seen as follows:
\begin{align*}
	\cP(\cdot,H_K<\infty)=\cP(\cdot, H_B<\infty,H_K<\infty)=Q^*_B(\cdot, H_K<\infty)=Q^*_K(\cdot).
\end{align*}
\end{proof}

 Having shown that the interlacement intensity $\nu$ is a quasi-process for a self-dual Markov process, a direct consequence of Hunt's time-reversibility theorem is the time-reversibility mentioned at the end of Section \ref{Sec2}:
 \begin{corollary}\label{co}
	If $\hat \nu$ is obtained from $\nu$ by time-reversal, then $\hat \nu=\nu$.
 \end{corollary}

We can finally turn back to random interlacements and give an alternative construction:

\begin{thm}
	Take a Poisson point process $\Xi$ whose intensity measure is two-sided Brownian motion started from Lebesgue measure $\lambda^d$ on $\R^d$. 
Restrict $\Xi$ on the paths being closest to the origin at time $T\in [0,1]$ and identify paths that are translations of each other. This yields a random interlacement for some level $\beta>0$.
\end{thm}
\begin{proof}
	This is a consequence of Theorem \ref{t} using the fact that restriction of the intensity turns into restriction of the Poisson point process. 
\end{proof}
 
\subsection*{Acknowledgement} The authors would like to thank Pat Fitzsimmons and Jay Rosen for drawing our attention to reference \cite{R}.

\end{document}